\documentclass[12pt,reqno]{amsart}

\usepackage{amssymb}
\usepackage[latin1]{inputenc}
\usepackage{amssymb}
\usepackage[all]{xy}
\usepackage{amsfonts}
\usepackage{float}
\usepackage{dsfont}
\usepackage[frenchb,english]{babel}
\usepackage{enumerate}
\usepackage{amsfonts, amsmath, amssymb, stmaryrd, latexsym}
\usepackage{icomma} 
\usepackage{pstricks}
\usepackage[all]{xy}
\usepackage{float}
\usepackage{amscd}
\usepackage{geometry}
\usepackage[frenchb,english]{babel}
\newtheorem*{thma}{Théorème A}
\newtheorem*{thmb}{Théorème B}
\theoremstyle{plain}
\newtheorem{thm}{Théorème}
\newtheorem{lem}[thm]{Lemme}

\newtheorem{cor}{Corollaire}
\theoremstyle{remark}
\newtheorem{rem}{Remarque}

\def\QQ{\mathbb{Q}}

\def\ZZ{\mathbb{Z}}
\addto\captionsfrench{}
\addto\captionsfrench{}
\begin{document}

\title{Capitulation des $2$-classes d'idéaux de type $(2, 4)$}

\author{Abedelmalek. Azizi}
\thanks{Recherche soutenue par l'Académie Hassan II des Sciences et
Techniques et l'URAC6, Maroc.}
\address{Département de Mathématiques\\
Faculté des Sciences\\
Université Mohammed 1\\
Oujda\\
Maroc}
\email{abdelmalekazizi@yahoo.fr}
\author{Mohammed. Taous}
\address{Département de Mathématiques, Faculté des Sciences et Techniques, Université Moulay Ismail, Errachidia, Maroc.}
\email{taousm@hotmail.com}



\maketitle

\begin{abstract}
In this paper, we establish two main results which give conditions
necessary and sufficient for a $ 2 $-group metabelian such that $G/G'$ is of type $(2, 4)$ either metacyclic or not. If $G$ is the Galois group of $\mathbf{k}_2^{(2)}/\mathbf{k}$  where $\mathbf{k}_2^{(2)}$ is the  the Second Hilbert $2$-class field of a number field $\mathbf{k}$, we will get results on the problem of capitulation.

\end{abstract}

\selectlanguage{francais}

\begin{abstract}
 Dans ce travail, nous allons établir deux résultats principaux qui donnent les conditions
nécessaires et suffisantes pour que un $2$-groupe métabélien tel que $G/G'$ est de type $(2, 4)$, soit métacyclique ou non. Si $G$ est le groupe de Galois $\mathbf{k}_2^{(2)}/\mathbf{k}$, où $\mathbf{k}_2^{(2)}$ le deuxième $2$-corps de classes de Hilbert d'un corps de nombre $\mathbf{k}$, nous obtiendrons des résultats sur le problème de capitulation.
\end{abstract}

\section{Introduction}
Soient $x$, $y$ et $z$ des éléments d'un groupe $G$. \textit{le
commutateur} de $x$ et de $y$ est l'élément: $$[x,\,
y]=x^{-1}y^{-1}xy=x^{-1}x^y.$$ On peut facilement montrer
les propriétés suivantes :
\begin{align}
[xy, z]& = [x, z]^y[y, z].\label{JNT4}\\
[x ,yz]& = [x, z][x, y]^z\label{JNT5}.
\end{align}
Si $X$ et $Y$ sont deux sous-ensembles de $G$, on désigne par
$[X,\,Y]$ le groupe engendré par les commutateurs $[x, y]$ où $x\in
X$  et $y\in Y$; et  $G'=[G, G]$ désignera \textit{le groupe dérivé}
(ou \textit{le groupe des commutateurs}) de $G$ et par $\gamma_i(G)$
le i-ème terme de la série centrale descendante de $G$ définie par
$\gamma_1(G)=G$ et $\gamma_{i+1}(G)=[\gamma_i(G), G]$. On dit que $G$
est \textit{nilpotent} s'il existe un entier $c$ tel que
$\gamma_{c+1}(G)=1$. Le plus petit entier qui vérifie cette égalité
est appelé \textit{classe de nilpotence}  de $G$. Il est bien connu
qu'un $p$-groupe $G$ est nilpotent. On appelle \textit{exposant}
d'un groupe abélien le plus grand des ordres de ses éléments. Rappelons qu'un groupe $G$ est
\textit{métabélien} si son groupe dérivé $G'=[G,
G]$ est abélien est \textit{métacyclique} si il possède
un sous-groupe cyclique normal $H$, tel que le quotient $G/H$ est
cyclique. notons aussi par $d(G)$, Le rang de $G$, c'est est le nombre minimal
de générateurs de $M$\\
\indent Dans la section \ref{JNT1}, nous allons établir deux résultats principaux (théorème \ref{2:015} p. \pageref{2:015} et théorème \ref{2:014} p. \pageref{2:014}),  qui donnent les conditions nécessaires et suffisantes pour que  un $2$-groupe métabélien tel que $G/G'$ est de type $(2, 4)$ (c'est-à-dire isomorphe à $\ZZ/2\ZZ\times\ZZ/4\ZZ$),
soit métacyclique ou non. cette nouvelle caractérisation est diffèrent à celle donner par N. Blackburn (théorème \ref{017}); elle va nous aider à étudier le problème de capitulation des $2$-classes d'idéaux de type $(2, 4)$. Le résultat principal est le théorème suivant:
\begin{thma}
Soient $G$ un $2$-groupe tel que $G/G'$ est de type $(2, 4)$ et $M$
le sous-groupe maximal de $G$ tel que $M/G'$ est de type $(2, 2)$.
Alors $M/M'$ est de type $(2, 2, 2)$ ou $(2,2^m)$. Plus précisément, les propriétés suivantes sont équivalentes:
\begin{enumerate}[\rm\indent 1.]
\item $G$ est métacyclique;
\item $M/M'$ est de type $(2, 2^m)$ avec $m\geq 1$;
\item $d(M)=2$.
\end{enumerate}
\end{thma}
\indent Soient $\mathbf{k}$ un corps de nombres de degré fini sur $\QQ$, $p$
un nombre premier, $\mathrm{C}_\mathbf{k}$ (resp. $\mathrm{C}_{\mathbf{k}, p}$) le
groupe (resp. le $p$-groupe) de classes de $k$. On note
$\mathbf{k}^{(1)}$ le corps de classes de Hilbert de $\mathbf{k}$ au
sens large. Soit  $\mathbf{k}^{(n)}$ (pour $n$ un entier naturel) la
suite de corps de classes de Hilbert définie par:
$\mathbf{k}^{(0)}=\mathbf{k}$ et
$\mathbf{k}^{(n+1)}=(\mathbf{k}^{(n)})^{(1)}$. Nous définissons de
la même façons la tour des $p$-corps de classes de Hilbert de
$\mathbf{k}$. Il suffit de remplacer les $\mathbf{k}^{(n)}$ par les
$\mathbf{k}^{(n)}_p$ et les corps par les $p$-corps. Soient
$\mathbb{M}$ une extension cyclique non ramifiée de $\mathbf{k}$,
$\mathrm{C}_{\mathbf{k},\mathbb{M}}$ le sous-groupe de $\mathrm{C}_{\mathbf{k}}$
associé à $\mathbb{M}$ par la théorie du corps de classes, et
$j_{\mathbf{k} \rightarrow \mathbb{M}}$ le morphisme de
$\mathrm{C}_{\mathbf{k}}$ vers $\mathrm{C}_{\mathbb{M}}$ qui fait correspondre à la
classe d'un idéal $\mathcal{A}$ de $\mathbf{k}$ l'idéal engendré par
$\mathcal{A}$ dans $\mathbb{M}$ et
$\mathcal{N}_{\mathbb{M}/\mathbf{k}}$ la norme de
$\mathbb{M}/\mathbf{k}$. Taussky dit que
\begin{enumerate}[\indent$\bullet$]
\item  M est de type $(A)$ $\Leftrightarrow$ $|\ker(j_{\mathbf{k} \rightarrow \mathbb{M}})\cap\mathcal{N}_{\mathbb{M}/\mathbf{k}}(\mathrm{C}_{\mathbb{M}})|>1$,
\item  M est de type $(B)$ $\Leftrightarrow$ $|\ker(j_{\mathbf{k}
\rightarrow
\mathbb{M}})\cap\mathcal{N}_{\mathbb{M}/\mathbf{k}}(\mathrm{C}_{\mathbb{M}})|=1$.
\end{enumerate}
D'après la théorie du corps de classes le noyau de $j_{\mathbf{k} \rightarrow \mathbb{M}}$ s'identifie au noyau du transfer ${\rm V}_{G\rightarrow H}: G/G'\rightarrow H/H'$ où $G=\mathrm{Gal}(\mathbf{k}_2^{(2)}/\mathbf{k})$ et $H=\mathrm{Gal}(\mathbb{M}_2^{(2)}/\mathbb{M})$. Pour calculer le noyau du transfer ${\rm V}_{G\rightarrow H}$, on va utiliser la formule suivante (\cite{Mi89}): Pour $g\in G$, si on pose $f=[\langle
g\rangle.H:H]$ et  $\{x_1, x_2,
\ldots,x_t\}$ un ensemble de représentants de $G/\langle
g\rangle H$, alors on a
\begin{equation}\label{JNT7}
{\rm V}_{G\rightarrow H}(gG')=\prod_{i=1}^tx_i^{-1}g^fx_i.H'.
\end{equation}
\indent Si le $2$-groupe de classes de $\mathbf{k}$ est de type $(2, 4)$, la structure de $G$ est classiquement peut déterminer dans des cas particuliers, en étudiant le type de capitulation de $\mathbf{k}$ dans les trois extensions quadratiques non ramifiées de $\mathbf{k}$. Pour cela il est nécessaire de trouver le système fondamentale de ces corps, ce qui n'est pas toujours possible. Mais dans plusieurs cas, le type de capitulation nous ne donne pas des informations sur  $G$. Comme $G/G'$ est de type $(2, 4)$, nous appliquons le théorème A et nous trouvons un nouveau critère utile pour déterminer le type de $G$:
\begin{thmb}
On garde les notations précédentes. Supposons que $G$ est d'ordre $>16$, alors les propriétés suivantes sont équivalentes:
\begin{enumerate}[\rm\indent 1.]
\item $G$ est métacyclique;
\item Le $2$-groupe de classes de $\mathbf{K_{3,  2}}$ est de type $(2, 2^m)$ avec $m\geq 3$;
\item Le $2$-nombre de classes de $\mathbf{K_{3,  2}}$ est $> 8$;
\item Le $2$-groupe de classes de $\mathbf{K_{1,  4}}$ est cyclique d'ordre $>2$.
\end{enumerate}
\end{thmb}
Contrairement à la méthode traditionnelle, nous avons caractériser la métacycliqulicité de $G$. Par exemple dans le travail de E. Benjamin, C. Snyder \cite{BeSn94}, si $|\ker(j_{\mathbf{k} \rightarrow \mathbf{K}_{i,  2}})\cap\mathcal{N}_{\mathbb{K}_{i,  2}/\mathbf{k}}(\mathrm{C}_{\mathbb{K}_{i,  2}})|=2$, pour chaque $\mathbb{K}_{i,  2}$ extension quadratique non ramifiée de $\mathbf{k}$ (on dit que la capitulation est de type $2A2A2A$), nous ne pouvons rien dire sur la structure de $G$. Pour nous si le $2$-rang du groupe de classes de $\mathbf{K_{3,  2}}$ est égal à $2$, nous obtiendrons que $G$ est métacyclique, puis nous calculons des noyaux de transfer, ce qui nous permet de résoudre le problème de capitulation d'une façons complète.
\section{Sur les $2$-groupes dont l'abélianisé est de type $(2, 4)$}
\label{JNT1}
Nous commençons par donner un théorème qui nous sera utile par la suite: N. Blackburn a montré dans \cite{Bl581} que :
\begin{thm} \label{2:005}
Soit $G$ un $p$-groupe. Si $G/G'\simeq (p^n, p^m)$ tel que $n\leq
m$, alors
\begin{enumerate}[\indent\rm 1.]
\item $\gamma_2(G)/\gamma_3(G)$ est cyclique d'ordre $\leq$ à $p^n$;
\item l'exposant de $\gamma_{i+1}(G)/\gamma_{i+2}(G)$ divise l'exposant de
$\gamma_i(G)/\gamma_{i+1}(G)$.
\end{enumerate}
\end{thm}
Si $G$ est un groupe non réduit à un élément, un sous-groupe $H$ de
$G$ est dit \textit{maximal} si le seul sous-groupe de $G$, distinct
de $G$ et contenant $H$ est $H$ lui-même. Le \textit{sous-groupe de
Frattini} $\Phi(G)$ d'un groupe G est l'intersection de tous ses
sous-groupes maximaux. Par exemple si $G$ est un $2$-groupe, alors
$\Phi(G)=G^2$ (\cite{Hu67}). Si
$G$ désigne un $2$-groupe d'ordre $2^n$ tel que
$\gamma_1(G)/\gamma_2(G)$ est de type (2, 4), la situation est schématisée par le
diagramme suivant :
\begin{figure}[H]
$$ \xymatrix{
   & \gamma_3(G) \ar@{->}[d] & \\
   & G'\ar@{->}[ld]\ar@{->}[d]\ar@{->}[rd]\\
  N_1\ar@{->}[rd]&\ar@{->}[ld] N_3\ar@{->}[d]\ar@{->}[rd]  & N_2\ar@{->}[ld]\\
  H\ar@{->}[rd]& M \ar@{->}[d]  & K\ar@{->}[ld]\\
  &G
  }
$$
\caption{\label{Fig1}}
\end{figure}

Nous allons montrer dans cette section des nouveaux résultats pour
ce groupe $G$, mais nous allons tout d'abord rappeler brièvement des
résultats pour le cas où $G'$ est cyclique non-trivial. Soit $c$ la
classe de nilpotence de G, alors
\begin{equation}
|G|=[G:G']\prod_{i=2}^{c}[\gamma_{i}(G):\gamma_{i+1}(G)].
\end{equation}
Le théorème précédent entraîne que
$[\gamma_{i}(G):\gamma_{i+1}(G)]=2$ pour $2\leq i\leq c$, par
conséquent $c=n-2$. Rappelons qu'un tel groupe qui vérifie cette
égalité est dite de \textit{classe presque maximale}. D'autre part,
C. Baginski et A. Konovalov, ont donné dans \cite{BK04} la liste
complète de ces groupes en fonction de leurs générateurs et suivant
certaines relations. parmi eux un théorème décrit tout les groupes $G$ qui sont métacycliques. Le nombre de groupes $G$ métacycliques d'ordre $2^n$,
$n\geq 5$ où $G/G'$ est de type $(2, 4)$ est égal à :$$\left\{
                    \begin{tabular}{ll}
                      $3$, & \text{si} $n=5$; \\
                      $4$, & \text{si} $n > 5$.
                    \end{tabular}
                  \right.$$
Ils sont donnés par la représentation suivante : $$ G_m=\langle a,
b :\, a^{2^{n-2}} = 1,\, b^4 = z_1,\, a^b = a^{-1}z_2\rangle,$$ où
$1\leq m \leq 4$ et les valeurs de $z_1$ et de $z_2$ sont données
dans le tableau suivant (pour  $G_{4}$ nous avons $n > 5$).
\begin{table}[H]
\begin{center}
\begin{small}
\begin{tabular}{ccccc}
\hline
        & $G_{1}$ & $G_{2}$      & $G_{3}$      & $G_{4}$           \\
\hline
 $z_1$  & 1        &  1            & $a^{2^{n-3}}$ & 1              \\
 $z_2$  & 1        &  $a^{2^{n-3}}$&  1            & $a^{2^{n-4}}$  \\
\hline
\end{tabular}
\caption{\label{Tabl3}}
\end{small}
\end{center}
\end{table}
Pour la preuve de ce résultat, l'auteur peut voir \cite{Ja75},
Théorème 5.3, p. 352. En fin rappelons que si $G$ un $2$-groupe
métacycliques d'ordre 16 tel que $G/G'$ est de type $(2, 4)$, alors $G$ est égal à
$$ M_{16}=\langle a ,b : a^8=b^2=1, a^b=a^{5}\rangle  \text{ ou } G_{1}=\langle a ,b : a^4=b^4=1, a^b=a^{-1}\rangle.$$
\subsection{Sur les $p$-groupe $G$ dont l'abélianisé est
de type $(p^n, p^m)$} Dans cette sous-section nous donne quelque
résultats sur les $p$-groupe $G$ dont l'abélianisé est de
type $(p^n, p^m)$ dûs à Blackburn \cite[p. 334 et 335]{Bl58}. Un
tel groupe $G$ est engendré par deux éléments $a$ et $b$
tels que $a^{p^n}\equiv b^{p^m}\equiv 1\mod \gamma_2(G)$
(Théorème de la base de Burnside).

\begin{thm}\label{017}
Un $p$-groupe $G$ est métacyclique si et seulement si
$G/\Phi(G')\gamma_3(G)$ est
métacyclique.
\end{thm}

\begin{lem}\label{2:021}
Si $G$ est un $p$-groupe non abélien engendré par deux
générateurs, alors $\Phi(G')\gamma_3(G)$ est
le seul sous-groupe maximal de $G'$ normal dans
$G$.
\end{lem}

\begin{lem}
Si $G$ est $p$-groupe non métacyclique tel que
$G/G'$ est de type $(p^n, p^m)$, alors
$G$ engendré par deux générateurs $a$ et $b$ et
$G/\Phi(G')\gamma_3(G)$ est
engendré par $a$ et $b$ modulo
$\Phi(G')\gamma_3(G)$ tel que
$$[a, b]=c,\,a^{p^n}\equiv b^{p^m}\equiv c^p\equiv [a, c]\equiv [b, c]\equiv 1\mod
\Phi(G')\gamma_3(G)$$
\end{lem}

\begin{cor} \label{JNT3}
Soit $G$ un $2$-groupe tel que $G/G'$ est de type $(2, 4)$. Alors
\begin{enumerate}[\indent\rm 1.]
  \item $G$ est engendré par deux éléments $a$ et $b$ tels que $a^{2}\equiv b^{4}\equiv 1\mod\gamma_2(G)$;
  \item $\Phi(G')\gamma_3(G)=\gamma_3(G)$;
  \item $G$ est métacyclique si et seulement $G/\gamma_3(G)$ est métacyclique;
  \item\label{2:024} Si $G$ est non métacyclique, alors $a^{2}\equiv b^{4}\equiv c^2\equiv [a, c]\equiv [b, c]\equiv
  1\mod\gamma_3(G)$ avec $[a, b]=c$.
\end{enumerate}
\end{cor}

\begin{proof} 1. c'est une conséquence immédiate du théorème de la base de
Burnside.\\
\indent 2. Le résultat est évident si $G$ est abélien. Supposons que
$G$ est non abélien, alors le théorème \ref{2:005}  montre que $[G':\gamma_3(G)]=2$, donc $\gamma_3(G)$
est un sous-groupe maximal de $G'$ normal dans $G$, ainsi le lemme \ref{2:021} implique que $\Phi(G')\gamma_3(G)=\gamma_3(G)$.\\
\indent 3. et 4. évident.
\end{proof}
E. Benjamin et C. Snyder ont donné dans \cite{BeSn94} un lemme qui
aide à vérifier si le $2$-groupe $G$ tel que
$G/G'$ est de type $(2, 2^m)$, est
métacyclique ou non($m>1$). Avant de continuer rappelons que  Le
\textit{groupe modulaire}  $M_{2^n}$ c'est un groupe d'ordre $ 2^n$
$(n>3)$ avec la représentation suivante:
$$\langle a, b : a^{2^{n-1}}=b^2=1, [a, b]=a^{2^{n-2}}\rangle.$$
Le groupe $M_{2^n}$ est métacyclique et
$M_{2^n}/\gamma_2(M_{2^n})\simeq (2, 2^{n-2}).$ En particulier
$\gamma_2(M_{2^n})$ est d'ordre deux. Si nous posons
$G^{(2, 2)}=\Phi(G)^2[G,
\Phi(G)]$, le lemme est le suivant:
\begin{lem}
Soit $G$ un $2$-groupe tel que
$G/G'$ est de type $(2, 2^m)$, où $m>2$.
Alors
\begin{enumerate}[\indent\rm i)]
\item si $G$ est abélien, alors $G/G^{(2,
2)}$ est de type $(2, 4)$;
\item si $G$ est modulaire, alors $G/G^{(2,
2)}$ est de type $(2, 4)$;
\item si $G$ est métacyclique non modulaire, alors $G/G^{(2,
2)}$ est l'unique groupe métacyclique non modulaire d'ordre $16$
dont l'abélianisé est de type $(2, 4)$;
\item si $G$ est non métacyclique, alors $G/G^{(2,
2)}$ est l'unique groupe non métacyclique d'ordre $16$ dont son
abélianisé est de type $(2, 4)$.
\end{enumerate}
\end{lem}
Si $G$ est un $2$-groupe tel que $G/G'$ est de type $(2, 4)$, le
lemme précédent c'est presque le corollaire précédent. La remarque
suivante montre cette coïncidence.
\begin{rem}\label{JNT2}
Soit $G$ un $2$-groupe tel que $G/G'$ est de type $(2, 4)$. Alors
$G^{(2, 2)}=\gamma_3(G)$.
\end{rem}
\begin{proof}
Soit $G=\langle x, y\rangle$  tel que $x^{2}\equiv y^{4}\equiv
1\mod\gamma_2(G)$. Remarquons d'abord que  $G^{(2, 2)}=(G^2)^2[G,
G^2]$. N. Blackburn  et L. Evensa  ont prouvé dans \cite[corollaire
2.3, p. 104]{BE79} que $G^{(2, 2)}=G^4\gamma_2^2(G)\gamma_3(G)$.
Comme $\gamma_2(G)/\gamma_3(G)$ est d'ordre $2$, alors
$\gamma_2^2(G)\subseteq \gamma_3(G)$, ainsi $G^{(2,
2)}=G^4\gamma_3(G)$.\\
\indent Si $G$ est un groupe abélien métacyclique d'ordre $16$, alors $G^4=1$ et la remarque est donc évidente. Supposons que $G$ est métacyclique d'ordre $>16$. Nous
pouvons montrer facilement que $\gamma_2(G)=\langle x^2\rangle$
(voir dans la suite le lemme \ref{2:013}), donc nous avons
$\gamma_3(G)=\langle x^4\rangle$, ce qui nous donne que
$\gamma_3(G)\subseteq G^4$ et $G^{(2, 2)}=G^4$. De plus nous savons
que $\gamma_2(G)/\gamma_3(G)$ est d'ordre $2$ et
$\gamma_3(G)\subseteq G^{(2, 2)}\subseteq \gamma_2(G)$, alors
$G^{(2, 2)}=\gamma_2(G)$ ou bien $G^{(2, 2)}=\gamma_3(G)$. Si nous
avons le premier cas, nous obtiendrons aussi que $G^4=\gamma_2(G)$.
Ce qui est impossible, car $x^4\in G^4$ et $x^4\notin
\gamma_2(G)$.\\
\indent Si $G$ est non métacyclique, alors $x^{2}\equiv y^{4}\equiv
1\mod\gamma_3(G)$ (Corollaire \ref{JNT3}). Alors $G^4\subseteq
\gamma_3(G)$, et par suite $G^{(2, 2)}=\gamma_3(G)$.
\end{proof}
\subsection{Résultats principaux}
L'idée de ces résultats est une conséquence d'une étude d'un cas particulier ou $G$ est groupe de Galois d'une certaine extension de corps de nombres(\cite{AzTa08}). Pour passer à un $2$-groupe quelconque les
travaux de E. Benjamin, F. Lemmermeyer, et C. Snyder et surtout
l'article \cite{BeLeSn07} m'ont beaucoup m'aider de démontrer mes conjectures. Nous reprenons les notations introduites
en \cite[Lemme 1]{BeLeSn97}. Soit $G=\langle a, b\rangle$ un $2$-groupe non métacyclique métabélien
 tel que $a^{2}\equiv b^{4}\equiv 1\mod\gamma_2(G)$. Les termes $c_i$ sont
définis comme suit: $[a, b] = c=c_2$ et $c_{j+1} = [b, c_j ]$. Nous
avons  $G' =\langle c_2, c_3, \ldots\rangle$, $\gamma_3(G) = \langle c_2^2, c_3,
\ldots\rangle$ et $\gamma_4(G) = \langle c_2^4, c_3^2, c_4, \ldots\rangle$ voir
\cite[Lemme 2]{BeLeSn97}. Comme $M$ est un sous-groupe maximal de
$G$ tel que $M/G'$ est de type $(2, 2)$, alors $M=\langle a, b^2,
G'\rangle$. Un calcul élémentaire utilisant le \cite[Lemme
2]{BeLeSn97} permet de vérifier que $M'\gamma_4(G)=\gamma_3(G)$. Par
\cite[Théorème 2.49ii]{Hal33}, nous avons $M'=\gamma_3(G)$. Ceci
nous permet d'énoncer le lemme suivant:

\begin{lem}\label{2:022} Soient $G$ un $2$-groupe non métacyclique métabélien tel que $G/G'$ est de type $(2, 4)$ et $M$ le sous-groupe maximal de $G$ tel
que $M/G'$ est de type $(2, 2)$. Alors $M'=\gamma_3(G)$.
\end{lem}

\begin{cor} Soient $G$ le seul $2$-groupe non métacyclique d'ordre $16$ tel que
$G/G'$ est de type $(2, 4)$ et $M$ le sous-groupe maximal de $G$ tel
que $M/G'$ est de type $(2, 2)$. Alors  $M/M'$ est de type $(2, 2,
2)$.
\end{cor}
\begin{proof} Comme $G$ est le seul $2$-groupe non métacyclique d'ordre $16$ tel que
$G/G'$ est de type $(2, 4)$ et $M$ le
sous-groupe maximal de $G$ tel que $M/G'$ est de type $(2, 2)$. Alors
$$G=\langle a ,b : c=[a, b], a^2=b^4=c^2=1, [a, c]=[b, c]=1\rangle.$$
Comme $G'=\langle c\rangle$, alors $G$ est métabélien et
$$M=\langle a, b^2, c\rangle\text{ \quad et \quad}M'=\langle [a, c], [a, b^2],[b^2, c]\rangle.$$
Les propretés \eqref{JNT4} et \eqref{JNT5} entraînent que
\begin{align*}
   [a, b^2]&=[a, b]\cdot[a, b]^b=c\cdot c^b=c^2[c,b]=1\\
   [b^2, c]&=[b, c]^b\cdot[b, c]=1.
\end{align*}
Par suite $M$ est un groupe abélien d'ordre 8; or $a^2=b^4=c^2=1$, alors $M=M/M'$ est de type $(2, 2, 2)$.
\end{proof}
Le théorème suivant généralise cette résultat à un $2$-groupe non métacyclique quelconque.
\begin{thm}\label{2:015}
Soient $G$ un $2$-groupe métabélien tel que $G/G'$ est de type $(2, 4)$ et $M$
le sous-groupe maximal de $G$ tel que $M/G'$ est de type $(2, 2)$.
Alors les propriétés suivantes sont équivalentes:
\begin{enumerate}[\rm\indent 1.]
\item $G$ est non métacyclique;
\item $M/M'$ est de type $(2, 2, 2)$;
\item $d(M)=3$.
\end{enumerate}
\end{thm}
\begin{proof}
Rappelons que si $d(\mathfrak{G})$ est le nombre minimal de générateurs
d'un groupe $\mathfrak{G}$ et $\mathfrak{M}$ est un sous-groupe d'indice fini de $\mathfrak{G}$, alors l'inégalité de Schreier est:
\begin{displaymath}
d({\mathfrak{M}})-1 \leq[\mathfrak{{G }}: \mathfrak{{M}}](d(\mathfrak{G})-1)~.
\end{displaymath}
Pour notre cas si $\mathfrak{G}=G=\langle a, b\rangle$  tel que
$a^{2}\equiv b^{4}\equiv 1\mod\gamma_2(G)$ et
$\mathfrak{M}=M=\langle a, b^2, \gamma_2(G)\rangle$, alors
$d(M)\in\{1, 2, 3\}$; or $M$ ne peut pas être cyclique car il
admet trois sous-groupes maximaux ($N_1$, $N_2$ et $N_3$ voir
Figure \ref{Fig1}). On conclut que $d(M)\in\{2, 3\}$.\\
\indent 1. $\Rightarrow$ 2. Supposons que $G$ est non métacyclique, alors le lemme \ref{2:022} nous montre que $M'=\gamma_3(G)$. Comme $G/G'$ est de type $(2, 4)$,
Le théorème \ref{2:005} donne que $\gamma_2(G)/\gamma_3(G)$ est d'ordre égal à $2$. Par le corollaire \ref{JNT3}, nous trouvons que $a^{2}\equiv b^{4}\equiv c^2\equiv 1\mod\gamma_3(G)$ avec $[a, b]=c$. Ceci montre que l'exposant de $M/M'$ est $2$. Remarquons que
$$
[M:M']=[M: \gamma_3(G)]=[M: \gamma_2(G)]\cdot[ \gamma_2(G): \gamma_3(G)]=4\cdot2=8
,$$ alors $M/M'$ est de type $(2, 2, 2)$.\\
\indent 2. $\Rightarrow$ 3. D'après théorème de la base de
Burnside.\\
\indent 3. $\Rightarrow$ 1. Si $d(M)=3$, alors $G$ ne peut pas être
métacyclique, car si $G$ est un sous-groupe d'un $p$-groupe métacyclique, alors nous avons $d(G)\leq 2$.
\end{proof}
On reprend les notations précédentes. Dans le cas métacyclique nous
avons le lemme suivant:
\begin{lem}\label{2:013} Soient $G$ un $2$-groupe métacyclique d'ordre $2^n$  tel que $n\geq 4$ et $G/G'$
est de type $(2, 4)$ et $M$ le sous-groupe maximal de $G$ tel que
$M/G'$ est de type $(2, 2)$. Alors $M'$ est d'ordre $\leq 2$. De
plus
\begin{displaymath}
M'=\left\{\begin{tabular}{ll}
            $1$, & \text{si} $G=G_{1}, G_{2}$, $G_{3}$  \text{ou} $G=M_{16}$ ; \\
            $\langle a^{2^{n-3}}\rangle$, & \text{si} $G=G_{4}$.
          \end{tabular}
   \right.
\end{displaymath}
\end{lem}
\begin{proof}
Supposons que $n\geq 4$ et $G$ est non modulaire, alors
\begin{displaymath}
G=G_m=\langle a, b :\, a^{2^{n-2}} = 1,\, b^4 = z_1,\, a^b = a^{-1}z_2\rangle,
\end{displaymath}
où $1\leq m \leq 4$ et les valeurs de $z_1$ et de $z_2$ sont
données dans le tableau \ref{Tabl3} ($n=4$ seulement pour
$G=G_{1}$, pour  $G_{4}$ nous avons $n > 5$). Comme $G$ est un
groupe métacyclique, alors $G'$ est cyclique, ce qui nous permet
d'écrire $G'=\langle[a, b]\rangle$. Calculons $[a, b]$:
\begin{align}
[a, b]&=a^{-1}a^b =a^{-2}z_2\nonumber \\
&=\left\{\begin{tabular}{ll}
                      $a^{-2}$, & \text{si} $G=G_{1}\text{ ou } G_{3}$ ; \\
                      $ a^{-2+2^{n-3}}$, & \text{si} $G=G_{2}$ ;\\
                      $ a^{-2+2^{n-4}}$, & \text{si} $G=G_{4}$.
                    \end{tabular}
          \right.\label{JNT6}\\
&=\left\{\begin{tabular}{ll}
                      $a^{-2}$, & \text{si} $G=G_{1} \text{ ou } G_{3}$ ; \\
                      $ a^{-2(1-2^{n-4})}$, & \text{si} $G=G_{2}$ ;\\
                      $ a^{-2(1-2^{n-5})}$, & \text{si} $G=G_{4}$.
                      \end{tabular}
          \right.\nonumber
\end{align}
Pour $G=G_{2}$, $n\geq 5$, alors $1-2^{n-4}$ est un nombre impair.
Pour $G=G_{4}$, $n\geq 6$, alors $1-2^{n-5}$ est un nombre impair.
Par conséquent on peut déduire que:
\begin{displaymath}
G'=\langle[a, b]\rangle=\langle a^{-2}\rangle=\langle a^2\rangle.
\end{displaymath}
Comme $a^{2}\equiv b^{4}\equiv 1\mod\gamma_2(G)$, alors $M=\langle a,
b^2\rangle$ et $M'=\langle[a, b^2]\rangle$. La propriété \eqref{JNT5} donne $[a, b^2]=[a, b][a, b]^b$. Calculons
$[a, b]^b$:
\begin{align*}
[a, b]^b&=\left\{\begin{tabular}{ll}
                      ${(a^{-2})}^b$, & \text{si} $G=G_{1}\text{ ou } G_{3}$ ; \\
                      ${( a^{-2+2^{n-3}})}^b$, & \text{si} $G=G_{2}$ ;\\
                      ${(a^{-2+2^{n-4}})}^b$, & \text{si} $G=G_{4}$.
                  \end{tabular}
          \right.\\
&=\left\{\begin{tabular}{ll}
                      $a^{2}$, & \text{si} $G=G_{1} \text{ ou } G_{3}$ ; \\
                      $ a^{(-1+2^{n-3})(-2-2^{n-3})}$, & \text{si} $G=G_{2}$ ;\\
                      $ a^{(-1+2^{n-4})(-2-2^{n-4})}$, & \text{si} $G=G_{4}$.
          \end{tabular}
          \right.\\
&=\left\{\begin{tabular}{ll}
                      $a^{2}$, & \text{si} $G=G_{1} \text{ ou } G_{3}$ ; \\
                      $ a^{2-2^{n-3}-2^{n-2}+2^{2n-6}}$, & \text{si} $G=G_{2}$ ;\\
                      $ a^{2-2^{n-4}-2^{n-3}+2^{2n-8}}$, & \text{si} $G=G_{4}$.
          \end{tabular}
          \right.
\end{align*}
La relation \eqref{JNT6} entraîne que
\begin{align*}
[a, b^2]&=\left\{\begin{tabular}{ll}
                      $1$, & \text{si} $G=G_{1}\text{ ou } G_{3}$ ; \\
                      $ a^{-2^{n-2}+2^{2n-6}}$, & \text{si} $G=G_{2}$ ;\\
                      $ a^{-2^{n-3}+2^{2n-8}}$, & \text{si} $G=G_{4}$.
          \end{tabular}
          \right.\\
&=\left\{\begin{tabular}{ll}
                      $1$, & \text{si} $G=G_{1} \text{ ou } G_{3}$ ; \\
                      $ a^{2^{n-2}(-1+2^{n-4})}$, & \text{si} $G=G_{2}$ ;\\
                      $ a^{2^{n-3}(-1+2^{n-5})}$, & \text{si} $G=G_{4}$.
          \end{tabular}
          \right.
\end{align*}
Comme $a^{2^{n-2}}=1$, alors  $M'$ est d'ordre $\leq 2$. De plus
$$M'=\left\{
                      \begin{tabular}{ll}
                      $1$, & \text{si} $G=G_{1}, G_{2}$  \text{ou} $G_{3}$; \\
                      $\langle a^{2^{n-3}}\rangle$, & \text{si} $G=G_{4}$.
                    \end{tabular}
                  \right.$$
Si $G$ est modulaire, alors $G=M_{16}=\langle a ,b : b^2=a^8=1,
a^b=a^{5}\rangle$. Il est facile de voir que $[a, b]=a^{4}$ et
$[a, b]^a=a^{4}$, ainsi $M'=\langle [a^2, b]\rangle=\langle
a^{8}\rangle=1$. Cela achève la preuve du lemme.
\end{proof}
\begin{thm} \label{2:014}
Soient $G$ un $2$-groupe non abélien tel que $G/G'$ est de type $(2, 4)$ et $M$
le sous-groupe maximal de $G$ tel que $M/G'$ est de type $(2, 2)$.
Alors les propriétés suivantes sont équivalentes:
\begin{enumerate}[\rm\indent 1.]
\item $G$ est métacyclique;
\item $M/M'$ est de type $(2, 2^m)$ avec $m\geq 2$;
\item $d(M)=2$.
\end{enumerate}
\end{thm}
\begin{proof}
1. $\Rightarrow$ 2. Soit $G$ un $2$-groupe métacyclique non modulaire d'ordre
$2^n$. Si $n\geq 4$ et  $G=G_{1}, G_{2}$  \text{ou} $G_{3}$,
puisque $M=\langle x, y^2\rangle$, le lemme précédent montre que\label{2:023}
$M/M'$ est de type $(2, 2^{n-2})$. Si $n\geq 6$ et $G=G_{4}$, alors
$M/M'$ est de type $(2, 2^{n-3})$. Si $G$ est le $2$-groupe modulaire, alors $G=\langle x, y: x^2=y^8=1, y^x=y^3\rangle$. Dans le lemme précédent nous avons signalé que $M=\langle x, y^2\rangle$ est un sous-groupe abélien, ce qui donne que $M/M'$ est
de type $(2, 2^2)$.\\
\indent 2. $\Rightarrow$ 3. D'après théorème de la base de Burnside.\\
\indent 3. $\Rightarrow$ 1. D'après le théorème précédent.
\end{proof}
Les deux théorèmes précédent montre que si $M/M'$ est d'ordre $>8$, alors $G$ est un $2$-groupe métacyclique, mais si $M/M'$ est d'ordre égal à $8$, alors $G$ est un $2$-groupe modulaire ou isomorphe à $G_{1}$, si $M/M'$ est de type $(2, 4)$ ou bien non métacyclique, si $M/M'$ est de type $(2, 2, 2)$. Ceci nous permet d'énoncer les deux corollaire suivants:
\begin{cor}\label{2:019}
Soient $G$ un $2$-groupe tel que $G/G'$ est de type $(2, 4)$ et $M$
le sous-groupe maximal de $G$ tel que $M/G'$ est de type $(2, 2)$ et
$M/M'$ est de type $(2, 4)$. Alors $G$ est un groupe modulaire ou
bien $$G=G_{1}=\langle x ,y : x^4=y^4=1, y^x=y^{-1}\rangle.$$
\end{cor}
\begin{cor}\label{2:017}
Soient $G$ un $2$-groupe tel que $G/G'$ est de type $(2, 4)$ et $M$
le sous-groupe maximal de $G$ tel que $M/G'$ est de type $(2, 2)$ et
$M/M'$ est d'ordre $> 8$. Alors $G$ est métacyclique non modulaire
non abélien.
\end{cor}
\begin{thm}\label{2:020}
Soient $G$ un $2$-groupe tel que $G/G'$ est de type $(2, 4)$ et $M$
(resp. $H$ et $K$) le sous-groupe maximal de $G$ tel que $M/G'$ est
de type $(2, 2)$ (resp. $H/G'$ et $K/G'$ sont cyclique d'ordre $4$).
Alors les propriétés suivantes sont équivalentes:
\begin{enumerate}[\rm\indent 1.]
\item $G$ est abélien;
\item $M/M'$ est de type ($2, 2)$;
\item $H$ est cyclique d'ordre $4$;
\item $K$ est cyclique d'ordre $4$.
\end{enumerate}
\end{thm}
\begin{proof}
1. $\Leftrightarrow$ 2. Théorème \ref{2:015} et Théorème \ref{2:014}.\\
\indent 1. $\Rightarrow$ 3. évident.\\
\indent 3. $\Rightarrow$ 1. Comme $[H:G']=4$ et $H$ est un groupe
d'ordre $4$, alors $G'=1$, ainsi $G$ est abélien.
\end{proof}
\section{Capitulation des 2-classes d'idéaux de type $(2, 4)$ d'un corps de nombres\label{009}}
Dans toute la suite de cette section $\mathbf{k}$ désigne un corps
de nombres tel que son $2$-groupe de classes est de type $(2, 4)$,
alors $G/G'$ est aussi de type $(2, 4)$, donc $G=\langle a,
b\rangle$ tel que $a^2\equiv b^4\equiv 1\mod G'$ et $\mathrm{C}_{\mathbf{k},
2}=\langle \mathfrak{c}, \mathfrak{d}\rangle\simeq \langle aG',
bG'\rangle$ où $(\mathfrak{c}, \mathbf{k}_2^{(2)}/\mathbf{k})=aG'$
et $(\mathfrak{d}, \mathbf{k}_2^{(2)}/\mathbf{k})=bG'$ avec $(\ .\ ,
\mathbf{k}_2^{(2)}/\mathbf{k})$ est le symbole d'Artin dans
$\mathbf{k}_2^{(2)}/\mathbf{k}$. Par suite, il existe trois
sous-groupes normaux de $G$ d'indice 2 : $H_{1, 2}$, $H_{2, 2}$ et
$H_{3, 2}$ tels que
\begin{center}
$H_{1, 2}=\langle b, G'\rangle$, $H_{2, 2}=\langle ab, G'\rangle$ et
$H_{3, 2}=\langle a, b^2, G'\rangle.$
\end{center}
\ \indent Il existe aussi trois sous-groupes normaux de $G$ d'indice
4 : $H_{1, 4}$,
 $H_{2, 4}$ et $H_{3, 4}$ tels que
\begin{center}
$H_{1, 4}=\langle a, G'\rangle$, $H_{2, 4}=\langle ab^2, G'\rangle$
et $H_{3, 4}=\langle b^2, G'\rangle.$
\end{center}
\ \indent Chaque sous-groupe $H_{i,j}$ de $\mathrm{C}_{\mathbf{k}, 2}$
correspond à une extension non ramifiée $\mathbf{K}_{i,j}$ de
$\mathbf{k}_2^{(2)}$ telle que $\mathrm{C}_{\mathbf{k}, 2}/H\simeq
\operatorname{Gal}(\mathbf{K}/\mathbf{k})$ et
$H_{i,j}=\mathcal{N}_{\mathbf{K}_{i,j}/\mathbf{k}}(\mathrm{C}_{\mathbf{k}_{i,j},
2})$. La situation est schématisée par le diagramme suivant :
\begin{figure}[H]
$$
\xymatrix{
   & \mathbf{k}_2^{(2)} \ar@{<-}[d] & \\
   & \mathbf{k}_2^{(1)}\ar@{<-}[ld]\ar@{<-}[d]\ar@{<-}[rd]\\
  \mathbf{K_{1,  4}}\ar@{<-}[rd]&\ar@{<-}[ld] \mathbf{K_{3,  4}}\ar@{<-}[d]\ar@{<-}[rd]  & \mathbf{K_{2,  4}}\ar@{<-}[ld]\\
  \mathbf{K_{1,  2}}\ar@{<-}[rd]& \mathbf{K_{3,  2}} \ar@{<-}[d]  & \mathbf{K_{2,  2}}\ar@{<-}[ld]\\
  &\mathbf{k}
  }
$$
\caption{\label{Fig2}}
\end{figure}
Notre but sera l'étude du problème de capitulation des $2$-classes
d'idéaux dans les extensions quadratiques non ramifiées
$\mathbf{K}_{1, 2}$, $\mathbf{K}_{2, 2}$ et $\mathbf{K}_{3, 2}$ et
aussi si possible dans les extensions de degrés $4$ non ramifiées:
$\mathbf{K}_{1, 4}$, $\mathbf{K}_{2, 4}$ et $\mathbf{K}_{3, 4}$.
Nous utilisons le symbole $(X_1, X_2, X_3)$ où $X_i\in\{4, 2, 2A,
2B\}$ pour $i=1$, $2$, $3$.  $X_i=4$ ou $2$ veut dire que quatre ou deux classes d'idéaux de $\mathbf{k}$ capitulent dans $\mathbf{K}_{i, 2}$. $X_i=2A$ (resp. $2B$) veut dire que deux classes d'idéaux de $\mathbf{k}$ capitulent dans $\mathbf{K}_{i, 2}$ et l'extensions quadratique  $\mathbf{K}_{i, 2}$ de $\mathbf{k}$ est de type $A$ (resp. $B$) avec $A$ et $B$ sont les conditions de Taussky.\\
\indent  On sait, d'après \cite{Ki76} que si $\mathbf{L}$ est un
corps de nombres et
$G=\operatorname{Gal}(\mathbf{L}_2^{(2)}/\mathbf{L})$ est
un $2$-groupe tel que $G/G'$ est de type $(2,
2)$, alors la structure de $G$ est complètement
déterminée en calculant le nombre des $2$-classes qui capitulent
dans les trois extensions quadratiques non ramifiées de $L$: $|\ker
j_i|$, $i=1$, 2, 3 et en précisant si les extensions non ramifiées
de $\mathbf{L}$ sont de type $A$ ou $B$. Plus précisément
$G$ est abélien de type (2, 2), quaternionique, diédral
ou semi-diédral; donc $G$ est toujours un 2-groupe
métacyclique. Mais dans le cas où $G/G'$ est
de type $(2, 2^m)$, avec $m>2$ E. Benjamin et C. Snyder ont proposé
dans \cite{BeSn94} une méthode pour savoir si $G$ est
métacyclique ou non, basée sur l'étude des mêmes questions pour de
$G/G^{(2, 2)}$ que pour $G$(voir
\cite{BeSn94} p. $12$). Si on note $\mathbf{L}_{1}$,
$\mathbf{L}_{2}$ et $\mathbf{L}_{3}$ les extensions quadratiques non
ramifiées de $\mathbf{L}$, on a le tableaux suivant :
$$\begin{array}{cccl}
  \hline
  \mathbf{L}_1 & \mathbf{L}_{2} & \mathbf{L}_{3}& G\\
  \hline
  4 & 4 & 4 & \hbox{abélien} \\
  2B & 2B & 2A & \hbox{modulaire ou non métacyclique} \\
  2A & 2A & 4 &\hbox{métacyclique}\\
   2A  & 2A  &  2A & \hbox{métacyclique ou non métacyclique}\\
   \hline
\end{array}$$
Dans les autres cas $G$ est non métacyclique. Alors si on
a le type $(2A, 2A, 2A)$ ou $(2B, 2B, 2A)$, nous ne pouvons rien
conclure. Dans \cite{Ko63}, H. Koch a caractérisé
$G/G^{(2, 2)}$, mais seulement dans le cas où
$\mathbf{L}$ est un corps quadratique imaginaire. Pour nous, nous
avons $G/G^{(2, 2)}=G/\gamma_3(G)$ (voir corollaire \ref{JNT3} et remarque \ref{JNT2}). Ceci montre que la dernière caractérisation coincide avec celle de N. Blackburn. Nous avons donné dans la Section \ref{JNT1} une nouvelle méthode pour savoir si $G$ est
métacyclique ou non, basé sur la structure du groupe abélien $H_{3,
2}/H_{3, 2}'$, c'est la structure du $2$-groupe de classes de
$\mathbf{K}_{3, 2}$.
\subsection{Le cas modulaire}
Dans cette section, nous étudions le problème de capitulation dans
le cas où $G$ est un $2$-groupe modulaire. Alors $G=\langle a ,b : a^2=b^8=1,
b^a=b^{5}\rangle$. Une application de la formule \eqref{JNT7} nous
donne le lemme suivant:
\begin{lem}
On garde les notations précédentes. Si $G$ est un $2$-groupe
modulaire, alors
\begin{enumerate}[\rm\indent 1.]
\item $\mathrm{V}_{G \rightarrow H_{i, 2}}(xG') =\left\{
            \begin{array}{ll}
              1, & \hbox{si $x=a$;} \\
              b^4, & \hbox{si $x=b^2$.}
            \end{array}
          \right.
$ avec $i=1$, $2$.
\item $\mathrm{V}_{G \rightarrow H_{3, 2}}(xG') =\left\{
            \begin{array}{ll}
              b^{-4}, & \hbox{si $x=a$;} \\
              b^4, & \hbox{si $x=b^2$.}
            \end{array}
          \right.
$
\item $\mathrm{V}_{G \rightarrow H_{i, 4}}(xG') =\left\{
            \begin{array}{ll}
              1, & \hbox{si $x=a$;} \\
              b^4, & \hbox{si $x=b$.}
            \end{array}
          \right.
$ avec $i=1$, $2$ ou $3$.
\end{enumerate}
\end{lem}
\begin{proof}
$G=\langle a ,b : a^2=b^8=1,
b^a=b^{5}\rangle$, alors $[b, a]=b^{-1}b^a=b^4$, ansi $G'=\langle b^4\rangle$ et $H_{1, 2}=\langle b\rangle$,  $H_{2, 2}=\langle ab, b^4\rangle$,  $H_{3, 2}=\langle a, b^2\rangle$,  $H_{1, 4}=\langle a, b^4\rangle$, $H_{2, 4}=\langle ab^2, b^4\rangle$,  et $H_{3, 4}=\langle b^2\rangle$. Comme $ba=ab^5$, $b^8=1$ et $a^2=1$, alors
\begin{align*}
    &(ab)^4=abababab=a^2b^5ba^2b^5b=b^{12}=b^4,\\
    &(ab^2)^2=abbab^2=abab^7=a^2b^{12}=b^4.
\end{align*}
On obtient que $
H_{2, 2}=\langle ab\rangle$ et $H_{2, 4}=\langle ab^2\rangle.$
On va montrer seulement le résultat 1. Pour les autres résultats, il
suffit de suivre la même méthode. Si $i$ égal à $1$ ou à $2$, la formule \eqref{JNT7} donnera que
$$
\mathrm{V}_{G \rightarrow H_{i, 2}}(xG') =\left\{
            \begin{array}{ll}
              a^2H'_{i, 2}, & \hbox{si $x=a$;} \\
              b^4[b^2, a]H'_{i, 2}, & \hbox{si $x=b^2$,}
            \end{array}
          \right.
$$
car $a\notin H_{i, 2}$ et $b^2\in H_{i, 2}$. Or $[b^2, a]=[b,
a]^b[b, a]=(b^4)^bb^4=b^8=1$ et $H_{i, 2}$ est un sous-groupe
cyclique de $G$, ce qui achève la preuve du résultat 1.
\end{proof}
\begin{thm}\label{003}
On garde les notations précédentes. Alors $G$ est un groupe
modulaire si et seulement si le $2$-groupe de classes de
$\mathbf{K_{3,  2}}$ est de type $(2, 4)$ et seulement deux classes qui capitulent dans $\mathbf{K_{3,  2}}$. Dans cette situation nous
avons:
\begin{enumerate}[\rm\indent 1.]
\item $\ker j_{\mathbf{k} \rightarrow \mathbf{K}_{i, 2}}=\{1, \mathfrak{c}\}$ avec $i=1$, $2$.
\item $\ker j_{\mathbf{k} \rightarrow \mathbf{K}_{3, 2}}=\{1, \mathfrak{c}\mathfrak{d}^2\}$.
\item $\ker j_{\mathbf{k} \rightarrow \mathbf{K}_{i, 4}}=\{1, \mathfrak{c}, \mathfrak{d}^2, \mathfrak{c}\mathfrak{d}^2\}$ avec $i=1$, $2$ ou $3$.
\item La capitulation des $2$-classes d'idéaux de $\mathbf{k}$ est de
type $(2B, 2B, 2A)$.
\end{enumerate}
\end{thm}
\begin{proof}
Supposons que le $2$-groupe de classes de $\mathbf{K_{3,  2}}$ est
de type $(2, 4)$, alors le groupe quotient $H_{3, 2}/H_{3, 2}'$ est
de type $(2, 4)$. D'après le corollaire \ref{2:019}, $G$ est un
groupe modulaire ou bien $$G=G_{1}=\langle a ,b : a^4=b^4=1,
a^b=a^{-1}\rangle.$$ Si $G=G_{1}$, alors $[a, b]=a^{-1}a^b=a^{-2}$
et $H_{3, 2}=\langle a, b^2\rangle$. $G'$ est un groupe cyclique
engendré par $a^2$, donc $H_{3, 2}'$ est aussi cyclique engendré par
$$[b^2, a]=[b, a]^b[b, a]=(a^{2})^ba^{2}=(a^b)^2a^2=a^{-2}a^2=1,$$
ceci prouve que $H_{3, 2}'=1$  et nous permet d'écrire:
$$
\mathrm{V}_{G \rightarrow H_{3, 2}}(xG') =\left\{
            \begin{array}{ll}
              a^2[a, b]=1, & \hbox{si $x=a$;} \\
              b^4[b^2, b]=1, & \hbox{si $x=b^2$.}
            \end{array}
          \right.
$$
Nous avons alors montré que $\ker \mathrm{V}_{G \rightarrow H_{i,
2}}=\{aG', b^2G', ab^2G', G'\}$, c'est-à-dire que quatre classes
capitulent dans $\mathbf{K_{3,  2}}$; c'est une contradiction avec
le faite que deux classes qui capitulent dans $\mathbf{K_{3, 2}}$,
alors $G$ est un groupe modulaire.

Inversement, si $G$ est un groupe modulaire, on a vu, dans la preuve
du théorème \ref{2:014}, que $H_{3, 2}/H_{3, 2}'$ est de type $(2,
4)$, alors le $2$-groupe de classes de $\mathbf{K_{3,  2}}$ est de
type $(2, 4)$, de plus le lemme précédent entraîne que $\ker
\mathrm{V}_{G \rightarrow H_{i, 2}}=\{ab^2, 1\}$, on conclut que deux classes capitulent dans
$\mathbf{K_{3,  2}}$; pour plus de précisions, on a $
\ker j_{\mathbf{k} \rightarrow \mathbf{K}_{3, 2}}=\{1, \mathfrak{c}\mathfrak{d}^2\}.
$ Pour 1. et 3., il suffit de calculer le $\ker$ de $\mathrm{V}_{G
\rightarrow H_{i, j}}$ en utilisant le lemme précédent. 4. c'est une conséquence de la formule \eqref{JNT7}.
\end{proof}

\subsection{Le cas métacyclique non modulaire}
Dans cette sous-section, nous étudions le problème de capitulation
dans le cas où $G$ est un $2$-groupe métacyclique non modulaire
d'ordre $2^n$ avec $n\geq 4$. Nous donnons ensuite le lemme que nous
avons utilisé pour calculer le groupe dérivé de $H_{i, j}$, puis
nous avons trouvé l'image de $aG'$, $bG'$ et $b^2G'$ par
$\mathrm{V}_{G \rightarrow H_{i, j}}$.
\begin{lem}
On garde les notations précédentes. Si $G=G_m$, alors
\begin{enumerate}[\rm\indent (i)]
\item $[a, b]=\left\{\begin{tabular}{ll}
                      $a^{-2}$, & \text{si} $G=G_{1} \text{ ou } G_{3}$ ; \\
                      $ a^{-2(1-2^{n-4})}$, & \text{si} $G=G_{2}$ ;\\
                      $ a^{-2(1-2^{n-5})}$, & \text{si} $G=G_{4}$.
          \end{tabular}
          \right.$
\item $[a, b^2]=\left\{\begin{tabular}{ll}
                      $1$, & \text{si} $G=G_{1}, G_{2} \text{ ou } G_{3}$ ; \\
                      $ a^{2^{n-3}(-1+2^{n-5})}$, & \text{si} $G=G_{4}$.
          \end{tabular}
          \right.$
\item $[a^2, b]=\left\{\begin{tabular}{ll}
                      $a^{-4}$, & \text{si} $G=G_{1}, G_{2} \text{ ou } G_{3}$ ; \\
                      $ a^{-4+2^{n-3}}$, & \text{si} $G=G_{4}$.
          \end{tabular}
          \right.$
\item $[a^2, b^2]=1$.
\end{enumerate}
\end{lem}
\begin{proof}
(i) et (ii). Voir la preuve du lemme \ref{2:013}\\
\indent (iii) Il suffit d'utiliser les égalités $[a^2, b]=[a, b]^a[a, b]$ et $[a, b]=a^{-2}z_2$.\\
\indent (iv) Nous avons $[a^2, b^2]=[a, b^2]^a[a, b^2]=\left\{\begin{tabular}{ll}
                      $ a^{2^{n-2}(-1+2^{n-5})}$, & \text{si} $G=G_{4}$; \\
                      $1$, & \text{si non}.
          \end{tabular}
          \right.$
Puisque $a^{2^{n-2}}=1$, le résultat est alors évident.
\end{proof}
\begin{cor}\label{012}
On garde les notations précédentes. Si $G=G_m$, alors
\begin{enumerate}[\rm\indent (i)]
\item $G'=\langle a^2\rangle$;
\item $H'_{2, 2}=H'_{1, 2}=\langle a^4\rangle$ et $H'_{3, 2}=\left\{
                      \begin{tabular}{ll}
                      $1$, & \text{si} $G=G_{1}, G_{2}$  \text{ou} $G_{3}$; \\
                      $\langle a^{2^{n-3}}\rangle$, & \text{si} $G=G_{4}$.
                    \end{tabular}
                  \right.$
\item $H'_{1, 4}=H'_{2, 4}=H'_{3, 2}=1$.
\end{enumerate}
\end{cor}
\begin{proof}
(i) Comme $G=\langle a, b\rangle$ est un groupe métacyclique, alors
son groupe des commutateurs $G'$ est cyclique et est engendré par
$[a, b]$. Pour $G=G_{4}$, $n\geq 6$, alors $1-2^{n-5}$ est un
nombre impair. Par conséquent on peut déduire que $
G'=\langle[a, b]\rangle=\langle a^{-2}\rangle=\langle a^2\rangle.$\\
\indent (ii) le résultat précédent nous permet de conclure que $H_{1, 2}=\langle b, a^2\rangle$, $H_{2, 2}=\langle ab, a^2\rangle$, $H_{3, 2}=\langle a, b^2\rangle$,  $H_{1, 4}=\langle a\rangle$,  $H_{2, 4}=\langle ab^2, a^2\rangle$,  et $H_{3, 4}=\langle a^2, b^2\rangle$. Finalement, il suffit d'utiliser les propriétés \eqref{JNT4} et \eqref{JNT5}, pour remarquer que
$$[ab^2, a^2]=[a, a^2]^{b^2}[b^2, a^2]=[b^2, a^2]\text{ et } [ab, a^2]=[a, a^2]^{b}[b, a^2]=[b, a^2].$$
Notre corollaire est donc une conséquence immédiate du lemme précédent.
\end{proof}
\begin{thm}\label{002}
On garde les notations précédentes. Supposons que $G$ est d'ordre $>16$, alors les propriétés suivantes sont équivalentes:
\begin{enumerate}[\rm\indent 1.]
\item $G$ est métacyclique;
\item Le $2$-groupe de classes de $\mathbf{K_{3,  2}}$ est de type $(2, 2^m)$ avec $m\geq 3$;
\item Le $2$-nombre de classes de $\mathbf{K_{3,  2}}$ est $> 8$;
\item Le $2$-groupe de classes de $\mathbf{K_{1,  4}}$ est cyclique d'ordre $>2$.
\end{enumerate}
\end{thm}
\begin{proof}
1. $\Leftrightarrow$ 2. $\Leftrightarrow$ 3. Conséquence directe du corollaire \ref{2:017}, p. \pageref{2:017}.\\
\indent 1. $\Rightarrow$ 4. Nous déduisons du corollaire précédent que $H_{1,  4}=\langle a\rangle$ et $a$ d'ordre $>2$,
alors $H'_{1,  4}=1$.  Compte tenu du $\langle a\rangle =H_{1,  4}/H'_{1,  4}\simeq \mathrm{C}_{2 K_{1,  4}}$, ce résultat est immédiat.\\
\indent 4. $\Rightarrow$ 1. Comme $H_{1,  4}/H'_{1,  4}\simeq \mathrm{C}_{2
K_{1,  4}}$ et $\mathrm{C}_{2 K_{1,  4}}$ est un groupe cyclique, le
théorème de la base de Burnside entraîne que $H_{1, 4}$ est un
sous-groupe d'ordre $>2$ normal cyclique de $G$ d'indice $4$.
Remarquons que $G/H_{1, 4}$ est un groupe cyclique, alors $G$ est
métacyclique.
\end{proof}
\begin{cor}
On garde les notations précédentes. Si $G$ est métacyclique  non abélien, alors la tour des $2$-corps de classes de Hilbert de $\mathbf{k}$ s'arrête en $\mathbf{k}_2^{(2)}$.
\end{cor}
\begin{proof}
Posons
$G=\mathrm{Gal}(\mathbf{k}_2^{(3)}/\mathbf{k}_2^{(1)})$,
alors le groupe des commutateurs $G'$ est égal à
$\mathrm{Gal}(\mathbf{k}_2^{(3)}/\mathbf{k}_2^{(2)})$, ainsi le
groupe quotient $G/G'$ est isomorphe à $G'$
le groupe des commutateurs de $G$, qui est cyclique, suivant le
corollaire précédent. Le théorème de la base de Burnside entraîne
que $G$ est un groupe cyclique et $G'=1$. Ceci
veut dire $\mathbf{k}_2^{(3)}=\mathbf{k}_2^{(2)}.$
\end{proof}
\begin{lem}\label{013}
On garde les notations précédentes. Si $G$ est un $2$-groupe
 métacyclique  non abélien, alors
\begin{enumerate}[\rm\indent 1.]
\item $\mathrm{V}_{G \rightarrow H_{i, 2}}(xG') =\left\{
            \begin{array}{ll}
              a^2H'_{i, 2}, & \hbox{si $x=a$;} \\
             H'_{i, 2}, & \hbox{si $x=b^2$.}
            \end{array}
          \right.
$ avec $i=1$, $2$.
\item $\mathrm{V}_{G \rightarrow H_{3, 2}}(xG') =\left\{
            \begin{array}{ll}
               H'_{3, 2}, & \hbox{si $x=a$ et $G=G_{1}$ ou $G_{3}$;}\\
               a^{2^{n-3}}H'_{3, 2}, & \hbox{si $x=a$ et $G=G_{2}$;}\\
               a^{2^{n-4}}H'_{3, 2}, & \hbox{si $x=a$ et $G=G_{4}$;}\\
              a^{2^{n-3}}H'_{3, 2}, & \hbox{si $x=b^2$ et $G=G_{3}$;}\\
              H'_{3, 2}, & \hbox{si $x=b^2$ et $G=G_{1}$, $G_{2}$ ou $G_{4}$.}
            \end{array}
          \right.
$
\item $\mathrm{V}_{G \rightarrow H_{i, 4}}(xG') =\left\{
            \begin{array}{ll}
             a^{2^{n-3}}H'_{i, 4}, & \hbox{si $x=a$ et $G=G_{4}$;}\\
             H'_{i, 4}, & \hbox{si $x=a$ et $G=G_{j}$  avec $j\neq4$;} \\
             H'_{i, 4}, & \hbox{si $x=b$ et $G\neq G_{3}$avec $i\neq 3$;} \\
             a^{2^{n-3}}H'_{i, 4}, & \hbox{si $x=b$ et $G=G_{3}$ avec $i\neq 3$;}\\
             H'_{3, 4}, & \hbox{si $x=b$ et $G=G_{1}$, $G_{2}$;}\\
             a^{2^{n-3}}H'_{3, 4}, & \hbox{si $x=b$ et $G=G_{3}$;}\\
             a^{2^{n-3}(-1+2^{n-5})}H'_{3, 4}, & \hbox{si $x=b$ et $G=G_{4}$.}
            \end{array}
          \right.
$ avec $i=1$, $2$ ou $3$.
\end{enumerate}
\end{lem}
\begin{proof}
Nous montrons ici le premier résultat, les autre se traitent de la
même façons. Nous avons $H_{1, 2}=\langle b, a^2\rangle$ et $H_{2,
2}=\langle ab, a^2\rangle$, alors $a\notin H_{i, 2}$, où $i=1$, $2$,
ce qui donne que $\mathrm{V}_{G \rightarrow H_{i, 2}}(aG')=
              a^2H'_{i, 2}.
              $ Maintenant, remarquons que $b^2\in H_{i, 2}$, où $i=1$, $2$; et $G/ H_{i, 2}=\langle aH'_{i, 2}\rangle$, alors
\begin{align*}
    \mathrm{V}_{G \rightarrow H_{i, 2}}(b^2G')&=b^4[b^2,a]H'_{i, 2}\\
                                              &=\left\{\begin{tabular}{ll}
                      $H'_{i, 2}$, & \text{si} $G=G_{1}, G_{2}$; \\
                      $ a^{-2^{n-3}}H'_{i, 2}$ & \text{si} $G=G_{3}$;\\
                      $ a^{-2^{n-3}(-1+2^{n-5})}H'_{i, 2}$, & \text{si} $G=G_{4}$.
          \end{tabular}
          \right.
\end{align*}
Si $G=G_{3}$, $G_{4}$, alors $n\geq 5$ et $H'_{i, 2}=\langle
a^4\rangle$; on conclut alors que les deux éléments $a^{-2^{n-3}}$
et $ a^{-2^{n-3}(-1+2^{n-5})}$ appartiennent à $H'_{i, 2}$.
Finalement on a:
$$
\mathrm{V}_{G \rightarrow H_{i, 2}}(b^2G')= H'_{i, 2}.
$$
\end{proof}
\begin{thm}\label{004}
On garde les notations précédentes. Si $G$ est métacyclique non modulaire non abélien, alors
\begin{enumerate}[\rm\indent 1.]
\item $\ker j_{\mathbf{k} \rightarrow \mathbf{K}_{i, 2}}=\{1, \mathfrak{d}^2\}$ avec $i=1$ ou $2$.
\item $\ker j_{\mathbf{k} \rightarrow \mathbf{K}_{3, 2}}=\left\{
                                                           \begin{array}{ll}
                                                             \{1, \mathfrak{c}, \mathfrak{d}^2, \mathfrak{c}\mathfrak{d}^2\}, & \hbox{si $G=G_{1}$;} \\
                                                             \{1, \mathfrak{d}^2\}, & \hbox{si $G=G_{2}$ ou $G_{4}$;} \\
                                                             \{1, \mathfrak{c}\}, & \hbox{si $G=G_{3}$.}
                                                           \end{array}
                                                         \right.
$
\item $\ker j_{\mathbf{k} \rightarrow \mathbf{K}_{i, 4}}=\left\{
                                                           \begin{array}{ll}
                                                             \mathrm{C}_{\mathbf{k}, 2}, & \hbox{si $G=G_{1}$ ou $G_{2}$;} \\
                                                             \{1, \mathfrak{c}, \mathfrak{d}^2, \mathfrak{c}\mathfrak{d}^2\}, & \hbox{si $G=G_{3}$;} \\
                                                            \{1, \mathfrak{d}, \mathfrak{d}^2, \mathfrak{d}^3\}, & \hbox{si $G=G_{4}$,}
                                                           \end{array}
                                                         \right.$ avec $i=1$, $2$.
\item $\ker j_{\mathbf{k} \rightarrow \mathbf{K}_{3, 4}}=\left\{
                                                           \begin{array}{ll}
                                                             \mathrm{C}_{\mathbf{k}, 2}, & \hbox{si $G=G_{1}$ ou $G_{2}$;} \\
                                                             \{1, \mathfrak{c}, \mathfrak{d}^2, \mathfrak{c}\mathfrak{d}^2\}, & \hbox{si $G=G_{3}$ ou $G=G_{4}$.}
                                                           \end{array}
                                                         \right.$
\item La capitulation des $2$-classes d'idéaux de $\mathbf{k}$ est de
type $(2A, 2A, 2A)$ ou bien $(2A, 2A, 4)$.
\end{enumerate}
\end{thm}
\begin{proof}
Nous montrons seulement le résultat 2. les autres se montrent de la
même façon. D'après le lemme précédent on a
$$\mathrm{V}_{G \rightarrow H_{3, 2}}(xG') =\left\{
            \begin{array}{ll}
               H'_{3, 2}, & \hbox{si $x=a$ et $G=G_{1}$ ou $G_{3}$;}\\
               a^{2^{n-3}}H'_{3, 2}, & \hbox{si $x=a$ et $G=G_{2}$;}\\
               a^{2^{n-4}}H'_{3, 2}, & \hbox{si $x=a$ et $G=G_{4}$;}\\
              a^{2^{n-3}}H'_{3, 2}, & \hbox{si $x=b^2$ et $G=G_{3}$;}\\
              H'_{3, 2}, & \hbox{si $x=b^2$ et $G=G_{1}$, $G_{2}$ ou $G_{4}$,}
            \end{array}
          \right.
$$
or le corollaire \ref{012}, montre que $H'_{3, 2}=1$, alors
$$\mathrm{V}_{G \rightarrow H_{3, 2}}(xG') =\left\{
            \begin{array}{ll}
               1, & \hbox{si $x=a$ et $G=G_{1}$ ou $G_{3}$;}\\
               a^{2^{n-3}}, & \hbox{si $x=a$ et $G=G_{2}$;}\\
               a^{2^{n-4}}, & \hbox{si $x=a$ et $G=G_{4}$;}\\
              a^{2^{n-3}}, & \hbox{si $x=b^2$ et $G=G_{3}$;}\\
             1, & \hbox{si $x=b^2$ et $G=G_{1}$, $G_{2}$ ou $G_{4}$.}
            \end{array}
          \right.
$$
Les définitions des $G_{2}$, $G_{3}$ et $G_{4}$, exigent que
$n\geq 5$, alors que $a^{2^{n-3}}$ et  $a^{2^{n-4}}$ sont deux
éléments différents de $1$. Ce qui nous permet de conclure que
$$\ker \mathrm{V}_{G \rightarrow H_{3, 2}}=\left\{
                                                           \begin{array}{ll}
                                                             \{G', aG', b^2G', ab^2G'\}, & \hbox{si $G=G_{1}$;} \\
                                                             \{G', b^2G'\}, & \hbox{si $G=G_{2}$ ou $G_{4}$;} \\
                                                             \{G', aG'\}, & \hbox{si $G=G_{3}$.}
                                                           \end{array}
                                                         \right.
$$
Par la loi de réciprocité d'Artin, nous trouvons le résultat 2.\\
\indent 5. les résultats 1. et 2. du lemme \ref{013} donnent:
$$
\ker{\rm V}_{G\rightarrow H_{i, 2}}\cap H_{i, 2}/G'=\{G', b^2G'\}\cap \langle bG'\rangle=\{G', b^2G'\}, \hbox{si $i=1$, $2$;}
$$ et
$
\ker{\rm V}_{G\rightarrow H_{3, 2}}\cap H_{3, 2}/G'=$
$$
                                          \left\{
                                            \begin{array}{ll}
                                             \langle aG', b^2G'\rangle\cap \langle aG', b^2G'\rangle=\langle aG', b^2G'\rangle, & \hbox{si $G=G_{1}$;} \\
                                             \{G', aG'\}\cap \langle aG', b^2G'\rangle=\{G', aG'\}, & \hbox{si $G=G_{3}$;}\\
                                             \{G', b^2G'\}\cap \langle aG', b^2G'\rangle=\{G', b^2G'\}, & \hbox{si non.} \\
                                            \end{array}
                                          \right.
$$
Par la loi de réciprocité d'Artin, nous trouvons que $$|\ker
j_{\mathbf{k} \rightarrow K_{i, 2}}\cap\mathcal{N}_{K_{i,
2}/\mathbf{k}}(C_{K_{i, 2}})|>1,$$ où $i=1$, $2$ et $3$. Les
résultats 1. et 2. de ce théorème entraînent que la capitulation des
$2$-classes d'idéaux de $\mathbf{k}$ est de type $(2A, 2A, 2A)$ ou
bien $(2A, 2A, 4)$.
\end{proof}
\begin{cor}\label{016}
On garde les notations précédentes. Si $G$ est métacyclique non modulaire non abélien
et seulement $\mathfrak{d}^2$ et son carré capitulent dans $\mathbf{K}_{3, 2}$, alors les propriétés suivantes sont équivalentes:
\begin{enumerate}[\rm\indent 1.]
\item $G=G_{2}$;
\item la tour des $2$-corps de classes de Hilbert de $\mathbf{K}_{3, 2}$ s'arrête en première étape;
\item $h(\mathbf{K}_{i, 4})=\dfrac{h(\mathbf{K}_{3, 2})}{2}$ avec $i=1$, $2$ ou $3$.
\end{enumerate}
\end{cor}
\begin{proof}
2. $\Leftrightarrow$ 3. Comme $G$ est un groupe métacyclique non modulaire non abélien et seulement $\mathfrak{d}^2$ et son carré capitulent dans $\mathbf{K}_{3, 2}$, alors $G$ est un groupe d'ordre $>16$ et le rang du $2$-groupe de classes de $\mathbf{K}_{3, 2}$ est égal à $2$ (théorème \ref{002}), alors la proposition 7 du \cite{BeLeSn98}, entraîne que la tour des $2$-corps de classes de Hilbert de $\mathbf{K}_{3, 2}$ s'arrête en première étape si, et seulement si  $h(\mathbf{K}_{i, 4})=\dfrac{h(\mathbf{K}_{3, 2})}{2}$ avec $i=1$, $2$ ou $3$.\\
1. $\Leftrightarrow$ 3. Puisque seulement la classe $\mathfrak{d}^2$
et son carré capitulent dans $\mathbf{K}_{3, 2}$, alors $G=G_{2}$
ou $G=G_{4}$; or si $2^n$ est l'ordre de $G$, nous avons d'une part
déjà montré (p. \pageref{2:023}) que
$$
H_{3, 2}/H_{3, 2}'\simeq\left\{
                     \begin{array}{ll}
                       \ZZ/2\ZZ\times \ZZ/2^{n-2}\ZZ, & \hbox{si $G=G_{2}$;} \\
                       \ZZ/2\ZZ\times \ZZ/2^{n-3}\ZZ, & \hbox{si $G=G_{4}$.}
                     \end{array}
                   \right.
$$
Ceci nous donne que
$$
\dfrac{h(\mathbf{K}_{3, 2})}{2}=\left\{
                                  \begin{array}{ll}
                                    2^{n-2}, & \hbox{si $G=G_{2}$;} \\
                                    2^{n-3}, & \hbox{si $G=G_{4}$.}
                                  \end{array}
                                \right.
$$
D'autre part on a pour $i=1$, $2$, $3$:
$$
h(\mathbf{K}_{i, 4})=[H_{i, 4}:H'_{i, 4}]=[H_{i, 4}:G'][G':H'_{i, 4}].
$$
D'après le corollaire \ref{012},  $H'_{1, 4}=H'_{2, 4}=H'_{3, 2}=1$ et $G'=\langle a^2\rangle$, alors $[G':H'_{i, 4}]$ est égal à $2^{n-4}$, l'ordre de $a^4$. Il est facile de voir que $[H_{i, 4}:G']=2$, ce qui implique que $h(\mathbf{K}_{i, 4})=2^{n-3}$. Finalement, $G=G_{2}$ si, et seulement si, $h(\mathbf{K}_{i, 4})=\dfrac{h(\mathbf{K}_{3, 2})}{2}$ avec $i=1$, $2$ ou $3$.
\end{proof}
\subsection{Le cas non métacyclique}
Dans cette sous-section, nous étudions le problème de capitulation
dans le cas où $G$ est un $2$-groupe non métacyclique.
\begin{thm}\label{005}
On garde les notations précédentes. Alors les propriétés suivantes sont équivalentes:
\begin{enumerate}[\rm\indent 1.]
\item $G$ est non métacyclique;
\item Le $2$-groupe de classes de $\mathbf{K_{3,  2}}$ est de type $(2, 2, 2)$;
\item Le $2$-rang du $2$-groupe de classes de $\mathbf{K_{3,  2}}$ est égal à $3$.
\end{enumerate}
\end{thm}
\begin{proof}
1. $\Leftrightarrow$ 2. $\Leftrightarrow$ 3. Conséquence directe du Théorème \ref{2:015}, p. \pageref{2:015}.
\end{proof}
\begin{lem}
On garde les notations précédentes. Si $G$ est un $2$-groupe non métacyclique, alors $$\mathrm{V}_{G \rightarrow H_{3, 2}}(xG')=\left\{
                                                                                \begin{array}{ll}
                                                                                  [a, b]\gamma_3(G), & \hbox{si $x=a$;} \\
                                                                                  \gamma_3(G), & \hbox{si $x=b^2$.}
                                                                                \end{array}
                                                                              \right.
$$
\end{lem}
\begin{proof}
Le résultat \ref{2:024} du corollaire \ref{JNT3}, p. \pageref{2:024}, donne que si $G$ est non métacyclique, alors
$$a^{2}\equiv b^{4}\equiv c^2\equiv [a, c]\equiv [b, c]\equiv
  1\mod\gamma_3(G),$$ avec $[a, b]=c.$
Si on applique la formule \eqref{JNT7}, nous trouverons que
$$\mathrm{V}_{G \rightarrow H_{3, 2}}(xG')=\left\{\begin{array}{ll}
                                                  a^2[a, b]H'_{3, 2}, & \hbox{si $x=a$;}\\
                                                  H'_{3, 2}, & \hbox{si $x=b^2$,}
                                                 \end{array}
                                                    \right.
$$
car $a\in  H_{3, 2}$ et $b^2\notin  H_{3, 2}$. Comme  $H'_{3, 2}=\gamma_3(G)$ (lemme \ref{2:022}, p. \pageref{2:022}) et $a^{2}\equiv  1\mod\gamma_3(G)$, alors
$$\mathrm{V}_{G \rightarrow H_{3, 2}}(xG')=\left\{
                                                                                \begin{array}{ll}
                                                                                  [a, b]\gamma_3(G), & \hbox{si $x=a$;} \\
                                                                                  \gamma_3(G), & \hbox{si $x=b^2$.}
                                                                                \end{array}
                                                                              \right.
$$
\end{proof}
\begin{cor}\label{006}
On garde les notations précédentes. Si $G$ est un $2$-groupe non métacyclique, alors
\begin{enumerate}[\rm\indent 1.]
\item $\ker j_{\mathbf{k} \rightarrow \mathbf{K}_{3, 2}}=\{1, \mathfrak{d}^2\}$;
\item La capitulation des $2$-classes d'idéaux de $\mathbf{k}$  dans $\mathbf{K}_{3, 2}$ est de type $2A$.
\end{enumerate}
\end{cor}
\begin{proof}
1. Soit $G=\langle a, b\rangle$  tel que $a^{2}\equiv b^{4}\equiv
1\mod\gamma_2(G)$. Les termes $c_i$ sont définis comme suit: $[a, b]
= c=c_2$ et $c_{j+1} = [b, c_j ]$. Nous avons  $G' =[c_2, c_3,
\ldots]$, $\gamma_3(G) = [c_2^2, c_3, \ldots]$ voir \cite[lemme
2]{BeLeSn97}. Ceci explique pourquoi $c=[a, b]\notin \gamma_3(G)$,
alors
$$
\ker\mathrm{V}_{G \rightarrow H_{3, 2}}= \{G', b^2G'\}.
$$
Par la loi de réciprocité d'Artin, nous trouvons
$$\ker j_{\mathbf{k} \rightarrow \mathbf{K}_{3, 2}}=\{1, \mathfrak{d}^2\}
$$
 \indent 2. Il suffit de remarquer que
$$\ker{\rm V}_{G\rightarrow H_{3, 2}}\cap H_{3, 2}/G'=\{G', b^2G'\}\cap \langle aG', b^2G', c^2G'\rangle.$$
Par la loi de réciprocité d'Artin, nous trouvons que $$|\ker
j_{\mathbf{k} \rightarrow K_{3, 2}}\cap\mathcal{N}_{K_{i,
2}/\mathbf{k}}(C_{K_{3, 2}})|>1.$$ Ce qui nous permet de  dire que
la capitulation des $2$-classes d'idéaux de $\mathbf{k}$  dans
$\mathbf{K}_{3, 2}$ est de type $2A$.
\end{proof}
\subsection{Le cas abélien}
Dans cette sous-section, nous avons deux théorèmes sur le problème
de capitulation, dans le cas où $G$ est un $2$-groupe abélien.
\begin{thm}\label{007}
On garde les notations précédentes. Alors les propriétés suivantes sont équivalentes:
\begin{enumerate}[\rm\indent 1.]
\item $G$ est  abélien;
\item Le $2$-groupe de classes de $\mathbf{K_{3,  2}}$ est de type $(2, 2)$;
\item Le $2$-nombre de classes de $\mathbf{K_{3,  2}}$ est $4$;
\item Le $2$-groupe de classes de $\mathbf{K_{i, 2}}$ est cyclique d'ordre $4$ avec $i=1$ ou $2$;
\item Le $2$-nombre de classes de $\mathbf{K_{i,  2}}$ est $4$ avec $i=1$ ou $2$;
\item La tour des $2$-corps de classes de Hilbert de $\mathbf{k}$ s'arrête en $\mathbf{k}_2^{(1)}$;
\item La capitulation des $2$-classes d'idéaux de $\mathbf{k}$ est de type $(4, 4, 4)$.
\end{enumerate}
\end{thm}
\begin{proof}
Par la loi de réciprocité d'Artin et le théorème \ref{2:020}, p. \pageref{2:020}, nous trouvons que 1. $\Leftrightarrow$ 2. $\Leftrightarrow$ 3. $\Leftrightarrow$ 4. $\Leftrightarrow$ 5.

1. $\Leftrightarrow$ 6. évident.

6. $\Leftrightarrow$ 7. voir \cite[lemme 1, p. 105]{Benj06}.
\end{proof}
\begin{cor}\label{008}
On garde les notations précédentes. Si $G$ est abélien, alors
\begin{enumerate}[\rm\indent 1.]
\item $\ker j_{\mathbf{k} \rightarrow \mathbf{K}_{i, 2}}=\{1, \mathfrak{c}, \mathfrak{d}^2, \mathfrak{c}\mathfrak{d}^2\}$ avec $i=1$, $2$ ou
$3$;
\item $\ker j_{\mathbf{k} \rightarrow \mathbf{K}_{i, 4}}=\mathrm{C}_{\mathbf{k}, 2}$ avec $i=1$, $2$ ou $3$.
\end{enumerate}
\end{cor}
\begin{proof}
Comme La tour des $2$-corps de classes de Hilbert de $\mathbf{k}$ s'arrête en $\mathbf{k}_2^{(1)}$, il suffit d'appliquer  \cite[théorème p.193]{Jan73}.
\end{proof}
\section{Application}
Dans cette section, on va donner deux applications de quelques
résultats précédents, l'un pour un corps quadratique réel, l'autre
pour un corps biquadratique imaginaire. Soit
$\mathbf{k}=\QQ(\sqrt{p_1p_2p_2})$ un corps quadratique réel tels
que les $p_i$ sont des nombres premier $\equiv 1\pmod 4$ et le
$2$-groupe de classes de $\mathbf{k}$ est de type $(2, 4)$.
Rappelons que dans \cite{Benj99} E. Benjamin n'à pas donné une
précision sur la structure du groupe
$G=\mathrm{Gal}(\mathbf{k}_2^{(2)}/\mathbf{k})$, si la capitulation
est de type $(2A , 2A, 2A)$. Pour nous, nous donnons une condition
nécessaire et suffisante pour que le groupe $G$ soit non
métacyclique, mais avant ça, établissons le lemme suivant:
\begin{lem}
Soient $p$, $q$ et $r$ des nombres premiers distincts
deux-à-deux tels que $p\equiv q\equiv r\equiv 1\pmod 4$, $\left(\dfrac{p}{q}\right)=1$ et $l$ est le $2$-rang du $2$-groupe de classes de $\QQ(\sqrt{p}, \sqrt{qr})$. Et nous avons les propriétés suivantes:
\begin{enumerate}[\rm\indent (1)]
  \item Si $\left(\dfrac{p}{r}\right)=1$, alors $l=\left\{
                                                  \begin{array}{ll}
                                                    3, & \hbox{si $\left(\dfrac{p}{q}\right)_4=\left(\dfrac{q}{p}\right)_4$ et $\left(\dfrac{p}{r}\right)_4=\left(\dfrac{r}{p}\right)_4$;} \\
                                                    2, & \hbox{si non.}
                                                  \end{array}
                                                \right.$
  \item Si $\left(\dfrac{p}{r}\right)=-1$, alors $l=\left\{
                                                  \begin{array}{ll}
                                                    2, & \hbox{si $\left(\dfrac{p}{q}\right)_4=\left(\dfrac{q}{p}\right)_4$;} \\
                                                    1, & \hbox{si non.}
                                                  \end{array}
                                                \right.$
\end{enumerate}
\end{lem}
\begin{proof}
C'est une consequence immediate du théorème 2 de \cite{AzMo01}, par exemple si $\left(\dfrac{p}{r}\right)=1$, nous trouvons que il existe $4$ idéaux premiers de $\QQ(\sqrt{p})$, qui se ramifient dans $\QQ(\sqrt{p}, \sqrt{qr})$, ainsi $l=3-e$, avec $e=0$ ou $1$; plus précisément on a: $e=0$ si, et seulement si, $\left(\dfrac{p}{q}\right)_4=\left(\dfrac{q}{p}\right)_4$ et $\left(\dfrac{p}{r}\right)_4=\left(\dfrac{r}{p}\right)_4$.
\end{proof}
\begin{thm}
Soient $p_1$, $p_2$ et $p_3$ des nombres premiers distincts
deux-à-deux tels que $p_1\equiv p_2\equiv p_3\equiv 1\pmod 4$. Si
le $2$-groupe de classes de $\mathbf{k}=\QQ(\sqrt{p_1p_2p_2})$ est
de type $(2, 4)$, alors le groupe
$G=\mathrm{Gal}(\mathbf{k}_2^{(2)}/\mathbf{k})$ est non métacyclique
si et seulement si
\begin{enumerate}[\rm\indent (1)]
  \item $\left(\dfrac{p_i}{p_j}\right)=\left(\dfrac{p_i}{p_k}\right)=1$ et
  \item $\left(\dfrac{p_i}{p_j}\right)_4=\left(\dfrac{p_j}{p_i}\right)_4$ et
$\left(\dfrac{p_i}{p_k}\right)_4=\left(\dfrac{p_k}{p_i}\right)_4$.
\end{enumerate}
\end{thm}
\begin{proof}
Nous connaissons que le $2$-groupe de classes de $\mathbf{k}=\QQ(\sqrt{p_1p_2p_2})$ est de type $(2, 4)$, alors $\mathbf{k}$ admet une extension cyclique de degré $4$ et non ramifiée pour les idéaux finis et infinis, donc cette extension doit être réelle, ce qui nous permet d'utiliser les résultats de \cite{Re34,Sc34}  et conclure que
\begin{equation}\label{014}
    \left(\dfrac{p_i}{p_j}\right)=\left(\dfrac{p_i}{p_k}\right)=1 \text{ et } \left(\dfrac{p_i}{p_jp_k}\right)_4=\left(\dfrac{p_jp_k}{p_i}\right)_4.
\end{equation}
Dans cette situation $\mathbf{K_{3,  2}}=\QQ(\sqrt{p_i}, \sqrt{p_jp_k})$. Le théorème \ref{005} entraîne que $G$ est non métacyclique si, et seulement si, le $2$-rang du $2$-groupe de classes de $\mathbf{K_{3,  2}}$ est égal à $3$. Le lemme précédent montre que c'est équivalent à $$\left(\dfrac{p_i}{p_j}\right)_4=\left(\dfrac{p_j}{p_i}\right)_4\text{ et }
\left(\dfrac{p_i}{p_k}\right)_4=\left(\dfrac{p_k}{p_i}\right)_4.$$
\end{proof}
Si $G$ est un groupe métacyclique, on n'a pas les derniers égalités. La relation \eqref{014} nous donne que \begin{equation}\label{015}
\left(\dfrac{p_i}{p_j}\right)_4=-\left(\dfrac{p_j}{p_i}\right)_4\text{ et }
\left(\dfrac{p_i}{p_k}\right)_4=-\left(\dfrac{p_k}{p_i}\right)_4.
\end{equation}

\begin{cor}
Soient $i$, $j$ et $k$ les entiers de la relation \eqref{014}. Si le $2$-groupe de classes de $\mathbf{k}=\QQ(\sqrt{p_1p_2p_2})$ est
de type $(2, 4)$ et le groupe $G=\mathrm{Gal}(\mathbf{k}_2^{(2)}/\mathbf{k})$
est  métacyclique, alors
\begin{enumerate}[\indent\rm(a)]
\item  $G$ est abélien si et seulement si $\left(\dfrac{p_j}{p_k}\right)=-1$.
\item  $G$ est non-abélien et non-modulaire si et seulement si $\left(\dfrac{p_j}{p_k}\right)=1$.
\end{enumerate}
\end{cor}
\begin{proof}
Si $i$, $j$ et $k$ sont les entiers de la relation \eqref{014} et si le $2$-groupe de classes de $\mathbf{k}=\QQ(\sqrt{p_1p_2p_2})$ est de type $(2, 4)$, alors d'après \cite{Ka76} on a deux $\left(\dfrac{p_j}{p_k}\right)=-1$ ou $\left(\dfrac{p_j}{p_k}\right)=1$. Dans le premier cas la tour des $2$-corps de classes de Hilbert de $\mathbf{k}$ s'arrête en $\mathbf{k}_2^{(1)}$ (\cite[théorème 1]{BeLeSn98}). Par conséquent $G$ est un groupe abélien. Si $\left(\dfrac{p_j}{p_k}\right)=1$, la relation \eqref{015} et le lemme précédent montre que le $2$-rang du $2$-groupe de classes de $\QQ(\sqrt{p_j}, \sqrt{p_ip_k})$ est égal à $2$. C'est-à-dire que $G$ admet un sous-groupe maximal de rang $2$, alors est non-abélien et non-modulaire.
\end{proof}
Dans la suite, nous donnons une autre application, cette fois nous prenons $\mathbf{k}=\QQ(\sqrt{2p}, i)$ et supposons que son $2$-groupe de classes est de type $(2, 4)$. Le théorème suivant est presque le résultats principal dans \cite{AzTa08}. Ici il devient une bonne application de résultats précédents.
\begin{thm}
Soit $\mathbf{k}=\QQ(\sqrt{2p}, i)$ avec $p$ un nombre premier tel
que $p\equiv 1\mod 8$ et $\left(\dfrac{2}{p}\right)_4=\left(\dfrac{p}{2}\right)_4=-1$, $2^n$ le
$2$-nombre de classes de $\QQ(\sqrt{-p})$.
Alors $G=\mathrm{G}al(\mathbf{k}_2^{(2)}/\mathbf{k}$ est un groupe
métacyclique non-modulaire, de plus $$G=\langle a, b :\, a^{2^{n}}
= 1,\, b^4 = 1,\, a^b = a^{-1+2^{n-1}}\rangle.$$
\end{thm}
\begin{proof}Notons que le $2$-groupe de classes de $\mathbf{k}$ est de type $(2, 4)$, car $p\equiv 1\mod 8$ et $\left(\dfrac{2}{p}\right)_4=\left(\dfrac{p}{2}\right)_4=-1$. Nous avons construit dans \cite{AzTa08} une extensions cyclique d'ordre 4, non ramifiée sur $\mathbf{k}$ et contient le corps $\mathbf{k}(\sqrt{2})$. Avec nos notations c'est le corps $\mathbf{K_{3,  2}}$. D'après \cite[théorème 9]{AzTa08}, le $2$-rang du $2$-groupe de classes de $\mathbf{K_{3,  2}}$ est égal à $2$, alors le théorème \ref{002}, $G$ est un groupe métacyclique non-modulaire d'ordre $2^{n+2}$, car le $2$-nombre de classes de $\mathbf{K_{3,  2}}$ est égal à $2^{n+1}$ avec $n\geq 3$, donc l'ordre de $G$ est divisible par $32$. Enfin le théorème 12, la proposition 9 du \cite{AzTa08} et le corollaire \ref{016}, impliquent que $$G=\langle a, b :\, a^{2^{n}}
= 1,\, b^4 = 1,\, a^b = a^{-1+2^{n-1}}\rangle.$$
\end{proof}
\bibliographystyle{alpha}

\end{document}